\documentclass[12 pt]{amsart}

\usepackage{hyperref}
\usepackage{etex}
\usepackage[shortlabels]{enumitem}
\usepackage{amsmath}
\usepackage{amsxtra}
\usepackage{amscd}
\usepackage{amsthm}
\usepackage{amsfonts}
\usepackage{amssymb}
\usepackage{eucal}
\usepackage[all]{xy}
\usepackage{graphicx}
\usepackage{tikz-cd}
\usepackage{mathrsfs}
\usepackage{subfiles}
\usepackage{mathpazo}
\usepackage[colorinlistoftodos, textsize=tiny]{todonotes}
\usepackage{morefloats}
\usepackage{pdfpages}
\usepackage{thm-restate}
\usepackage[utf8]{inputenc}
\usepackage{epigraph}
\usepackage{csquotes}
\usepackage[margin=1.5in]{geometry}
\usepackage{adjustbox}
\usepackage{scalerel}
\usepackage{stackengine}
\stackMath
\newcommand\reallywidehat[1]{%
\savestack{\tmpbox}{\stretchto{%
  \scaleto{%
    \scalerel*[\widthof{\ensuremath{#1}}]{\kern-.6pt\bigwedge\kern-.6pt}%
    {\rule[-\textheight/2]{1ex}{\textheight}}
  }{\textheight}%
}{0.5ex}}%
\stackon[1pt]{#1}{\tmpbox}%
}
\graphicspath{ {images/} }

\RequirePackage{color}
\definecolor{myred}{rgb}{0.75,0,0}
\definecolor{mygreen}{rgb}{0,0.5,0}
\definecolor{myblue}{rgb}{0,0,0.65}

\usepackage{color}

\usepackage{hyperref}
\hypersetup{citecolor=blue}
\usepackage{tikz}
\usetikzlibrary{matrix,arrows,decorations.pathmorphing}

\theoremstyle{plain}
\newtheorem{theorem}[subsection]{Theorem}

\newtheorem{proposition}[subsection]{Proposition}
\newtheorem{lemma}[subsection]{Lemma}
\newtheorem{corollary}[subsection]{Corollary}

\theoremstyle{definition}
\newtheorem{definition}[subsection]{Definition}
\newtheorem{remark}[subsection]{Remark}

\newtheorem{conjecture}[subsection]{Conjecture}

\theoremstyle{remark}
\newtheorem{notation}[subsection]{Notation}

\numberwithin{equation}{section}
\newcommand\nc{\newcommand}
\nc\on{\operatorname}
\nc\renc{\renewcommand}

\newcommand*{\shom}{\mathscr{H}\kern -.5pt om}
\newcommand*{\stor}{\mathscr{T}\kern -.5pt or}
\newcommand*{\sext}{\mathscr{E}\kern -.5pt xt}

\makeatletter
\providecommand\@dotsep{5}
\renewcommand{\listoftodos}[1][\@todonotes@todolistname]{%
\@starttoc{tdo}{#1}}
\makeatother

\makeatletter
\newcommand{\customlabel}[2]{\protected@write \@auxout {}{\string \newlabel {#1}{{#2}{\thepage}{#2}{#1}{}} }\hypertarget{#1}{#2}}

\renewcommand\hom{\mathrm{Hom}}

\DeclareMathOperator\rk{rk}

\DeclareMathOperator\Mod{Mod}

\DeclareFontFamily{U}{wncy}{}
\DeclareFontShape{U}{wncy}{m}{n}{<->wncyr10}{}
\DeclareSymbolFont{mcy}{U}{wncy}{m}{n}
\DeclareMathSymbol{\Sha}{\mathord}{mcy}{"58}

\makeatletter
\@namedef{subjclassname@2020}{%
  \textup{2020} Mathematics Subject Classification}
\makeatother

\def\listtodoname{List of Todos}
\def\listoftodos{\@starttoc{tdo}\listtodoname}

\title[Applications around the Putman-Wieland conjecture]{Applications of the algebraic geometry of the Putman-Wieland conjecture}
\author{Aaron Landesman and Daniel Litt}
\AtEndDocument{\bigskip{\footnotesize%
  \textsc{Department of Mathematics, MIT, 
182 Memorial Dr,   
\mbox{Cambridge, MA 02142}} \par
Aaron Landesman: \texttt{aaronl@mit.edu} \par
    \addvspace{\medskipamount}
  \textsc{Department of Mathematics, University of Toronto, 
	  Bahen Centre,
40 St. George St.,
  \mbox{Toronto, Ontario, Canada, M5S 2E4}}  \par
  Daniel Litt: \texttt{daniel.litt@utoronto.ca} \par
  }}

\usepackage{microtype}
\begin{document}

\begin{abstract}
We give two applications of our prior work toward the Putman-Wieland
conjecture.
First,
we deduce a strengthening of a result of 
Markovi\'c-To\v{s}i\'c on virtual mapping class group actions on the homology of covers. 
Second, let $g\geq 2$ and let $\Sigma_{g',n'}\to \Sigma_{g, n}$ 
be a finite $H$-cover of topological surfaces. 
We show the virtual action of the mapping class group of $\Sigma_{g,n+1}$ on an $H$-isotypic component of $H^1(\Sigma_{g'})$ has non-unitary image. 
\end{abstract}

\maketitle

\section{Introduction}
\subsection{Review of the Putman-Wieland conjecture}
\label{subsection:putman-wieland-statement}

We aim to explain applications of our prior work 
\cite{landesmanL:canonical-representations}
toward the Putman-Wieland conjecture, and so we begin by reviewing the statement
of the Putman-Wieland conjecture.

Let $\Sigma_{g,n}$ denote an orientable topological surface of genus $g$
with $n$ punctures. Let $H$ be a finite group.
Given a finite unramified $H$-cover of  topological surfaces $\Sigma_{g',n'} \to \Sigma_{g,n}$,
there is an action of a finite index subgroup $\Gamma$ of the mapping class group $\Mod_{g,n+1}$ of $\Sigma_{g,n+1}$ on
$H_1(\Sigma_{g'}, \mathbb C)$, as we now explain.
The group $\on{Mod}_{g,n+1}$ acts on $\pi_1(\Sigma_{g,n},x)$ for some basepoint $x$, and we can take $\Gamma$ to be the stabilizer of the surjection
$\phi: \pi_1(\Sigma_{g,n},x) \twoheadrightarrow H$, where 
$\phi$ corresponds to the cover $\Sigma_{g',n'} \to \Sigma_{g,n}$.
Then, for $x' \in \Sigma_{g',n'}$ mapping to $x$, $\Gamma$ acts on
$\pi_1(\Sigma_{g',n'},x')=\ker \phi$, preserving the conjugacy classes of the loops around the punctures in $\Sigma_{g', n'}$, and hence acts on $H_1(\Sigma_{g'},\mathbb C)$.

\begin{conjecture}[Putman-Wieland, \protect{\cite[Conjecture 1.2]{putmanW:abelian-quotients}}]
\label{conjecture:putman-wieland-intro}
Fix $g \geq 2$ and $n \geq 0$.
For any unramified cover 
$\Sigma_{g',n'} \to \Sigma_{g,n}$,
the vector space $H_1(\Sigma_{g'}, \mathbb C)$ has no 
nonzero vectors with finite orbit under the action of $\Gamma$.
\end{conjecture}

Note that the Putman-Wieland conjecture is false when $g=2$, see
\cite{markovic}.
We next describe applications of our prior work relating to the
Putman-Wieland conjecture.

\subsection{The results of Markovi\'c-To\v{s}i\'c}
\label{subsection:markovic-tosic}

In \cite[Theorem 1.5]{markovic2}, Markovi\'c-To\v{s}i\'c verified some new cases
of the Putman-Wieland conjecture.
We are able to recover their results from ours. In fact, we are able to deduce a
slight generalization of their results, as explained in
\autoref{remark:slight-generalization}.

We now formally define what it means for $f: X \to Y$ to furnish a
counterexample to Putman-Wieland.
\begin{definition}
	\label{definition:pw-counterexample}
	Suppose $h: \Sigma_{g',n'} \to \Sigma_{g,n}$
is a finite covering of topological surfaces.
Let $\Gamma \subset \on{Mod}_{g,n+1}$
denote the finite index subgroup preserving $h$.
The action of $\Gamma$ on $\pi_1(\Sigma_{g',n'}, \mathbb C)$ induces an action on
$H_1(\Sigma_{g',n'}, \mathbb C)$ which preserves the subspace spanned by homology classes
of loops around punctures, and hence also induces an action on
$H_1(\Sigma_{g'}, \mathbb C) \simeq H^1(\Sigma_{g'}, \mathbb C)$, via Poincar\'e duality.
We say $h$ {\em furnishes a counterexample to Putman-Wieland}, if
there is some nonzero $v \in H^1(\Sigma_{g'}, \mathbb C)$ with
finite orbit under $\Gamma$.

Let $Y$ be a compact Riemann surface of genus $g$ with $n$ marked points $p_1,
\ldots, p_n$.
Upon identifying $\pi_1(Y-\{p_1, \ldots, p_n\}) \simeq \pi_1(\Sigma_{g,n})$,
let $f: X \to Y$ be the
covering of compact Riemann surfaces, ramified at most over $p_1, \ldots, p_n$,
corresponding to the topological cover $h$.
If $h$ furnishes a counterexample to Putman-Wieland, we also say $f: X \to Y$ furnishes a counterexample to
Putman-Wieland.

Suppose $h$ is Galois with Galois group $H$, and $\rho$ is an irreducible
$H$-representation.
Note that $H$ and $\Gamma$ simultaneously act on $H^1(\Sigma_{g'}, \mathbb C)$.
We say a counterexample to Putman-Wieland is {\em $\rho$-isotypic} 
if every element of the the $\rho$-isotypic subspace $H^1(\Sigma_{g'}, \mathbb C)^\rho$ has finite orbit under $\Gamma$. Here if $V$ is an $H$-representation, $V^\rho$ denotes the $\rho$-isotypic subspace of $V$, i.e.~the image of the natural evaluation map $$\hom_H(\rho, V)\otimes \rho\to V.$$
\end{definition}
\begin{remark}
\label{remark:}
Even though the Putman-Wieland conjecture assumes $g \geq 2$, we still say $f: X
\to Y$ furnishes a counterexample to Putman-Wieland when $Y$ has genus $g \leq 1$.
\end{remark}

Here is the slight improvement on \cite[Theorem 1.5]{markovic2}, which we will
prove in \autoref{proof-mt}. We let $\lambda_1(X)$ denote the smallest nonzero eigenvalue of the Laplacian acting on
$L^2$ functions on
$X$.
\begin{theorem}
	\label{theorem:markovicT}
	Suppose $Y$ has genus $g \geq 2$ and $f: X \to Y$
	furnishes a counterexample to Putman-Wieland. Then 
$\frac{1}{g-1} \geq 2 \lambda_1(X)$.
If $f$ is Galois, then 
$\frac{1}{g-1} > 2 \lambda_1(X)$.
\end{theorem}
\begin{remark}
	\label{remark:slight-generalization}
	\autoref{theorem:markovicT} is slightly stronger than
\cite[Theorem 1.5]{markovic2}
as we now explain.
In their result, they first choose $\varepsilon$ and then choose $g$ subject to
the inequality 
$g \geq (1+\epsilon)/\epsilon$, or equivalently $\varepsilon \geq
\frac{1}{g-1}$.
They conclude $\varepsilon \geq \lambda_1(X)$. (Though their argument in fact
gives
$\varepsilon \geq 2\lambda_1(X)$, as they accidentally omitted a factor of $2$.)
In general, we also find this same inequality. However, when the cover is moreover Galois, 
using
\autoref{theorem:markovicT}, we find
$\varepsilon \geq \frac{1}{g-1}
> 2 \lambda_1(X)$, 
so we obtain a strict inequality instead of a weak inequality.
\end{remark}

Recall that the gonality of a curve $X$ is the smallest degree of a non-constant map $f: X
\to \mathbb P^1$. We use $\on{gon}(X)$ to denote the gonality of $X$.
In order to prove \autoref{theorem:markovicT}, the first step is the same as in
\cite{markovic2}, as we reduce to a statement about the gonality of $X$ using the Li-Yau
inequality and then apply the following proposition, which may be of independent
interest.

\begin{proposition}
	\label{proposition:gonality}
Suppose $Y$ has genus $g$ and
$f: X \to Y$ is a cover furnishing a counterexample to Putman-Wieland. Then,
$\on{gon}(X) \leq \deg f$.
When $f$ is additionally Galois, we have $\on{gon}(X) \leq \deg f - (g-1)$.
\end{proposition}

We prove this in \autoref{proof:gonality}.
The non-Galois case of the above proposition is proven via different
methods in \cite{markovic2}. 

\subsection{Prill's problem}
\label{subsection:intro-prill}

Along the way to our proof of \autoref{theorem:markovicT}, we make a brief digression regarding
Prill's problem \cite[Chapter VI, Exercise D]{ACGH:I} in \autoref{remark:prill-solution}. 
There, we explain why
a general genus $2$ curve has a cover such that every fiber moves in a pencil.
As detailed in \autoref{remark:companion-paper},
we prove a stronger result in the shorter and less technical article \cite{landesmanL:prill},
but include these observations here as it is a simple consequence of the tools
we develop to recover the results of Markovi\'c-To\v{s}i\'c described above.

\subsection{Isotypicity}
\label{subsection:isotypic}

We next prove a result ruling out $\rho$-isotypic counterexamples to
Putman-Wieland in genus at least $2$.
\begin{theorem}
	\label{theorem:isotypic}
Suppose $f: X\to Y$ is a Galois $H$-cover furnishing a counterexample to 
Putman-Wieland which is $\rho$-isotypic in the sense of
\autoref{definition:pw-counterexample}. 
Then $Y$ has genus at most $1$. 

Equivalently, if the genus of $Y$ is at least $2$, for each irreducible $H$-representation $\rho$, there exists an element of $H^1(X, \mathbb{C})^\rho$ with infinite orbit under the virtual action of the mapping class group of $Y$.
\end{theorem}

This will follow from a stronger statement:
\begin{theorem}\label{theorem:non-unitary}
Let $X\to Y$ be a Galois $H$-cover, where $Y$ has genus at least $2$. Let $\rho$ be an irreducible complex $H$-representation. Then the virtual action of the mapping class group of $Y$ on $H^1(X, \mathbb{C})^\rho$ is not unitary.	
\end{theorem}
\autoref{theorem:isotypic} is immediate, as representations with finite image are unitary.

We prove \autoref{theorem:non-unitary} in \autoref{subsection:isotypic-proof}.
The idea is to investigate a certain bilinear pairing, arising from the derivative of a period map and studied earlier in \cite{landesmanL:canonical-representations}. 

A recent paper of Boggi-Looijenga \cite{boggiL:curves-with-prescribed-symmetry},
which has since been retracted by the authors,
claims to show that in some settings the Jacobian of the generic $H$-curve has
$\mathbb{Q}$-endomorphism algebra $\mathbb{Q}[H]$. While their argument is
incomplete, we show the claimed result would imply the 
Putman-Wieland conjecture in \autoref{corollary:boggi-looijenga}. 

In \cite{markovic}, Markovi\'{c} gives a counterexample to
Putman-Wieland $f: X \to Y$, where $Y$ has genus $2$.
\autoref{theorem:isotypic} immediately implies the following.
\begin{corollary}
	\label{corollary:no-isotypic}
	Let $f' : X' \to Y$ 
	be the Galois closure of the counterexample to Putman-Wieland $f: X \to
	Y$ from
	\cite{markovic}.
For any irreducible $H$ representation $\rho$, $f'$
	is not a
	$\rho$-isotypic counterexample to Putman-Wieland.
\end{corollary}
\begin{remark}
	\label{remark:non-isotypic}
	\autoref{corollary:no-isotypic} is especially interesting, as all the counterexamples where $Y$ has genus
$0$ or $1$ which we know of are $\rho$-isotypic for some $\rho$.
This example and other examples of covers which factor through this one
are the only ones we know of which are not
$\rho$-isotypic. It would be quite interesting to produce other counterexamples
to Putman-Wieland which are not $\rho$-isotypic for any $\rho$.
\end{remark}

\subsection{Overview}

In this paper, we give two algebro-geometric consequences of our work on the
Putman-Wieland conjecture.
In \autoref{section:background}, we recall background on parabolic bundles and
notation for versal families.
In \autoref{section:globally-generated}, we connect non-generically globally
generated vector bundles to the Putman-Wieland conjecture; this connection is
a straightforward consequences of our past work, but we spell it out for
completeness.
In \autoref{section:bilinear-pairing}, we discuss a certain isotypicity property
which would imply the Putman-Wieland conjecture, and we prove \autoref{theorem:isotypic}.
Finally, in \autoref{section:markovic-tosic} we give a proof of a slight improvement of the main result of \cite{markovic2}, \autoref{theorem:markovicT}.

\subsection{Acknowledgements}
We would like to thank
an anonymous referee,
Marco Boggi,
Anand Deopurkar,
Joe Harris, 
Neithalath Mohan Kumar,
Eric Larson,
Rob Lazarsfeld,
Eduard Looijenga,
Vladimir Markovi\'{c},
Anand Patel,
Andy Putman,
Will Sawin,
Ravi Vakil,
and
Isabel Vogt
for helpful discussions related to this paper.
Landesman was supported by the National Science Foundation under Award No.
DMS-2102955. Litt was supported by NSF Grant DMS-2001196.
This material is based upon work supported by the Swedish Research Council under
grant no. 2016-06596 while the authors were in residence at Institut Mittag-Leffler in Djursholm, Sweden during the fall of 2021.

\section{Background and notation}
\label{section:background}

Throughout, we work over the complex numbers, unless otherwise stated.
For a pointed finite-type scheme or Deligne-Mumford stack $(X, x)$ over
$\mathbb{C}$, we will use $\pi_1(X, x)$ to denote the topological fundamental group 
 of the associated complex-analytic space or analytic stack.
In the remainder of this section, we recall notation for parabolic bundles, so
that we can state a relevant proposition on vector bundles which are not
generically globally generated.
For more detail, we recommend the reader consult
\cite[\S2]{LL:geometric-local-systems}.
We also review notation for versal families, which we also used in
\cite{landesmanL:canonical-representations}.

\subsection{A lightning review of parabolic bundles}
\label{subsection:parabolic-review}

Fix a smooth proper connected curve $C$ over an arbitrary field and let $D = x_1 + \cdots + x_n$ be a reduced divisor on
$C$. 
Recall that a {\em parabolic bundle} on $(C,D)$ is a vector bundle $E$ on $C$, a
decreasing filtration $E_{x_j} = E_j^1 \supsetneq E_j^2 \supsetneq \cdots
\supsetneq E_j^{n_j+1} =
0$ for each $1 \leq j \leq n$, and an increasing sequence of real numbers $0\leq
\alpha_j^1<\alpha_j^2<\cdots<\alpha_j^{n_j}<1$ for each $1 \leq j \leq n$.
referred to as {\em weights}. 
We use $E_\star = (E, \{E^i_j\}, \{\alpha^i_j\})$ to denote the data of a parabolic bundle.

Given a parabolic bundle
$E_\star = (E, \{E^i_j\}, \{\alpha^i_j\})$,
let $J \subset \{1, \ldots, n\}$ denote the set of 
integers $j \in \{1, \ldots, n\}$ for which $\alpha^1_j = 0$, and define
\begin{align}
	\label{equation:coparabolic}
\widehat{E}_0 : = \ker( E \to \oplus_{j \in J} E_{x_j}/E_j^2).
\end{align}
(This is a special case of more general notation used for coparabolic bundles
as in \cite[2.2.8]{LL:geometric-local-systems} or the equivalent
\cite[Definition 2.3]{bodenY:moduli-spaces-of-parabolic-higgs-bundles}, but is all we will need for this
paper.)
In particular, $\widehat{E}_0 \subset E$ is a subsheaf.

We next make sense of a relative variant of the above notions. Namely,
let $\mathscr C \to \mathscr B$ be a relative smooth proper curve with
geometrically connected fibers and let $\mathscr D \subset \mathscr C\to
\mathscr B$ be a relative \'etale Cartier divisor. Then, a relative parabolic bundle on
$(\mathscr C, \mathscr D)$ is a vector bundle $\mathscr E$ on $\mathscr C$,
a decreasing filtration $\mathscr E|_\mathscr D = \mathscr E^1 \supsetneq
\mathscr E^2
\supsetneq \cdots \supsetneq \mathscr E^{m+1} = 0$, and an increasing sequence of real numbers
$0 \leq \alpha^1 < \alpha^2 < \cdots < \alpha^m < 1$ referred to as weights.
Above, the $\mathscr E^i$ are sheaves supported on $\mathscr D$.
We use $\mathscr E_\star = (\mathscr E, \{\mathscr E^i\}, \{\alpha^i\})$ to
denote the data of a relative parabolic bundle.
\begin{remark}
	\label{remark:}
Note that if $(\mathscr E, \{\mathscr E^i\}, \{\alpha^i\})$ is a relative
parabolic bundle on $(\mathscr C, \mathscr D)$, and $(C, D)$ is the fiber over a
point $b \in \mathscr B$, with $D = x_1 + \cdots + x_n$, then we can recover the
data of a parabolic bundle $(E, \{E^i_j\}, \{\alpha^i_j\})$ 
over a field, in the above sense, as follows.
Take $E := \mathscr
E|_C$ as the underlying vector bundle. Take the filtration at $x_j$ to be the
restriction of the filtration $\mathscr E^i$ to $x_j$, with repetitions removed.
Finally, take $\alpha^i_j:= \max_k \{ \alpha^k : \mathscr E^k|_{x_j} =
E^i_j\}$.
\end{remark}
To conclude our treatment of the relative setting, we define
\begin{align}
	\label{equation:relative-coparabolic}
	\widehat{\mathscr E}_0 : = \begin{cases}
		\ker( \mathscr E \to \mathscr E|_{\mathscr D} / \mathscr E^2) & \text{ if }
		\alpha^1= 0, \\
		 \mathscr E & \text{ otherwise.} \\
	\end{cases}
\end{align}

Returning to the absolute setting, 
parabolic
bundles admit a notion of \emph{parabolic stability}, analogous to the usual
notion of stability for vector bundles, which we next recall. 
First, the {\em parabolic degree} of a parabolic bundle $E_\star$ is
\begin{align*}
	\on{par-deg}(E_\star) := \deg(E) + \sum_{j=1}^n \sum_{i=1}^{n_j}
	\alpha^i_j \dim(E_j^i/E_j^{i+1}).
\end{align*}
Then, the {\em parabolic slope} is defined by $\mu_\star(E_\star) :=
\on{par-deg}(E_\star)/\rk(E_\star)$.
Any subbundle $F \subset E$ has an induced parabolic structure $F_\star \subset
E_\star$ defined as
follows:
the filtration
over $x_j$ on $F$ is obtained
	from the filtration $$F_{x_j} = E^1_j \cap F_{x_j} \supset E^2_j \cap
	F_{x_j} \supset \cdots \supset E^{n_j+1}_j \cap F_{x_j} = 0$$
	by removing redundancies. For the weight associated to $F^i_j \subset
	F_{x_j}$ one takes $$\max_{\substack{k, 1 \leq k \leq n_j}} \{ \alpha^k_j : F^i_j = E^k_j \cap
	F_{x_j}\}.$$
A parabolic bundle $E_\star$ is {\em parabolically semi-stable} if for every
nonzero subbundle $F \subset E$ with
induced parabolic structure $F_\star$, we have $\mu_\star(F_\star) \leq
\mu_\star(E_\star)$.

Mehta and Seshadri 
\cite[Theorem 4.1, Remark 4.3]{mehta1980moduli},
give a correspondence between
\emph{parabolically stable parabolic bundles of parabolic degree zero} on
$(C,D)$ and irreducible unitary local systems on $C\setminus D$.
This bijection sends a local system $\mathbb{V}$ to the Deligne canonical extension of $(\mathbb{V}\otimes \mathscr{O}, \on{id}\otimes d)$ with the parabolic structure induced by the connection (as in \cite[Definition 3.3.1]{LL:geometric-local-systems}).

For our later results, 
it will also be useful to have a lower bound on the rank of certain vector
bundles.
This follows from the following general result about non-generically generated
vector bundles, whose proof ultimately relies on Clifford's theorem for vector
bundles.
The reader may take this as a black box.
For a less technical variant of this statement, see 
\cite[Proposition 6.3.1]{LL:geometric-local-systems}. 
\begin{proposition}[{ \cite[Proposition 6.3.6]{LL:geometric-local-systems}}]
	\label{proposition:generic-parabolic-global-generation}
	Suppose $C$ is a smooth proper connected genus $g$ curve and
	$E_\star = (E, \{E^i_j\}, \{\alpha^i_j\})$ is a nonzero parabolic bundle $C$
	with respect to $D = x_1 + \cdots + x_n$.
		Suppose $E_\star$ is parabolically semistable. 
	Let $U \subset \widehat E_0$ be a (non-parabolic) subbundle with
		$c := \rk E - \rk U$ and
		$\delta := h^0(C, \widehat{E}_0) - h^0(C, U)$. 
	\begin{enumerate}
		\item[(I)] If $\mu_\star({E}_\star)> 2g-2 + n$, 
		then $\rk E >
			gc - \delta$.
		\item[(II)] If $\mu_\star({E}_\star)= 2g-2+n$, then $\rk
			E \geq gc - \delta$.
	\end{enumerate}
	In particular, if $\widehat{E}_0$ fails to be generically globally
	generated, and $\mu_\star({E}_\star) \geq 2g-2+n$, $\rk E \geq g.$
\end{proposition}
\begin{remark}
\label{remark:}
The statement of \autoref{proposition:generic-parabolic-global-generation} is equivalent to \cite[Proposition
6.3.6]{LL:geometric-local-systems}, but differs slightly in that we write
``$E_\star$ is parabolically stable'' in place of ``$\widehat{E}_\star$ is
coparabolically stable'' and $\mu_\star(E_\star)$ in place of
$\mu_\star(\widehat{E}_\star)$.
However, by definition 
$E_\star$ is parabolically stable if and only if $\widehat{E}_\star$ is
coparabolically stable and $\mu_\star(E_\star) = \mu_\star(\widehat{E}_\star)$
\cite[Definitions 2.2.9 and 2.4.2]{LL:geometric-local-systems}.
Finally, the final ``In particular,\dots'' statement is an immediate consequence
of $(II)$.
\end{remark}

We next introduce the notation $E^\rho_\star$ as the parabolic bundle
corresponding to a representation $\rho$.
\begin{notation}
	\label{notation:rep-to-vector-bundle}
	Let $Y$ be a curve and $D \subset Y$ a divisor.
Recall that under the Mehta-Seshadri correspondence \cite{mehta1980moduli}, there is a bijection
between irreducible representations of $\pi_1(Y-D)$ and parabolic degree $0$ stable
parabolic vector bundles on $Y$, with parabolic structure along $D$.
Given an irreducible $H$ representation $\rho$, we use $E^\rho_\star$ to denote
the parabolic bundle corresponding to the representation
$\pi_1(Y-D) \simeq \pi_1(\Sigma_{g,n}) \xrightarrow{\phi} H \xrightarrow{\rho}
\on{GL}_{\dim \rho}(\mathbb C)$.
\end{notation}

\subsection{Notation for versal families}
\label{subsection:versal}

We set some notation to describe families of covers of curves.
\begin{notation}
	\label{notation:versal-family}
	We fix non-negative integers $(g,n)$ so that $n \geq 1$ if $g = 1$ and
	$n \geq 3$ if $g = 0$, i.e., $\Sigma_{g,n}$ is hyperbolic.
	Let $\mathscr{M}$ be a connected complex variety. A \emph{family
	of $n$-pointed curves of genus $g$ over $\mathscr{M}$} is a smooth
	proper morphism $\pi: \mathscr{C}\to \mathscr{M}$ of relative dimension one, with geometrically
	connected genus $g$ fibers, equipped with $n$ sections $s_1,\cdots,
	s_n: \mathscr{M}\to \mathscr{C}$ with disjoint images. Call such a family \emph{versal} if the
	induced map $\mathscr{M}\to \mathscr{M}_{g,n}$ is dominant and \'etale.
	Here $\mathscr M_{g,n}$ denotes the Deligne-Mumford moduli stack of $n$-pointed genus
	$g$ smooth proper curves with geometrically connected fibers.

	If $\pi: \mathscr{C}\to \mathscr{M}$ is a family of $n$-pointed curves, we let
	$\mathscr D := \coprod_{i=1}^n \on{im}(s_i)$ denote the union of the images of
	the sections, which is finite \'etale of degree $n$ over $\mathscr M$.
	Also let
	$\mathscr{C}^\circ:=\mathscr{C}\setminus \bigcup_i \on{im}(s_i)$,
	let $j: \mathscr C^\circ \hookrightarrow \mathscr{C}$ be the natural
	inclusion, and let $\pi^\circ := \pi \circ j : \mathscr C^\circ \to \mathscr M$ denote the composition. We will refer to $\pi^\circ: \mathscr{C}^\circ\to \mathscr{M}$ as the \emph{associated family of punctured curves}. If $\pi^\circ$ arises as the family of punctured curves associated to a versal family of $n$-pointed curves, we will call it a \emph{punctured versal family}.
	We will frequently use $m \in \mathscr M$ as a basepoint, and $c \in
	\mathscr C^\circ$ as a basepoint with $\pi^\circ(c) = m$.
	\begin{equation}
		\label{equation:}
		\begin{tikzcd}
			\mathscr C^\circ \ar {r}{j} \ar {rd}{\pi^\circ} &
			\mathscr C \ar {d}{\pi} & \mathscr D \ar{l} \\
			& \mathscr M \ar{ur}{s_1} \ar[bend
			right=70,swap,ur,"s_n"]
			\arrow[bend right = 60, ur, draw=none, "\ddots"]
		 &
		\end{tikzcd}
	\end{equation}
\end{notation}

We also will need the notion of a versal family of $\phi$-covers.
\begin{definition}
	\label{definition:versal-phi}
	Specify a surjection
	$\phi: \pi_1(\Sigma_{g, n}, v_0)\twoheadrightarrow H$, where $H$ is a finite group.
	A {\em versal family of $\phi$-covers} is the data of 
	\begin{enumerate}
\item a dominant \'etale morphism $\mathscr{M}\to \mathscr{M}_{g,n}$, with
$\pi^\circ: \mathscr{C}^\circ\to \mathscr{M}$ the associated punctured versal family,
\item a point $c\in \mathscr{C}^\circ$, $m=\pi^\circ(c)$, and an identification $i: \pi_1(\Sigma_{g, 0, n}, v_0)\simeq \pi_1(\mathscr{C}_m^\circ, c)$, and
\item a finite \'etale Galois $H$-cover $f: \mathscr{X}^\circ\to
	\mathscr{C}^\circ$ inducing a map
	$$\pi_1(\mathscr{C}_m^\circ, c) \to \pi_1(\mathscr{C}^\circ,
	c)/\pi_1(\mathscr{X}^\circ,x) \simeq H$$
	agreeing with the surjection $\phi$ under the
	identification of $(2)$.
\end{enumerate}
\end{definition}
We use $\mathscr X$ to denote the normalization of $\mathscr C$ in the
function field of $\mathscr X^\circ$, so $\mathscr X$ is the relative smooth proper curve
compactifying $\mathscr X^\circ$.

\section{Connection to vector bundles which are not generically globally
generated}\label{section:globally-generated}

We next set up notation to relate the Putman-Wieland conjecture to a certain
period map.

\begin{notation}
	\label{notation:period-map}
	Fix a surjection $\phi: \pi_1(\Sigma_{g, n}, v_0)\twoheadrightarrow H$, where $H$ is a finite group.
	Let $\mathscr C^\circ \to \mathscr M,  \mathscr{X}^\circ\to \mathscr{C}^\circ$ be a versal family of $\phi$-covers,
	as in \autoref{definition:versal-phi}.
	Let $\rho$ be an irreducible $H$-representation and let
	$\mathbb U^\rho$ denote the local system on $\mathscr C^\circ$ with
	monodromy representation given by $\pi_1(\mathscr C^\circ) \to H
	\xrightarrow{\rho} \on{GL}_{\dim \rho}(\mathbb C)$, where the first
	map is induced by the $H$-cover $\mathscr X^\circ \to \mathscr C^\circ$.

	Let $(\mathscr E^\rho_\star, \nabla^\rho)$
	denote the Deligne canonical extension with its natural parabolic structure
	(the Deligne canonical extension is defined in \cite[Definition
	4.1.2]{LL:geometric-local-systems}; the parabolic structure is defined in \cite[Definition 3.3.1]{LL:geometric-local-systems} for curves, but the definition there generalizes naturally to families)
	of
	$(\mathbb U^\rho \otimes_{\mathbb C} \mathscr O_{\mathscr C^\circ},
	\on{id} \otimes d)$ to $\mathscr C$.
	Let $\mathbb W^\rho := R^1 \pi^\circ_* (\mathbb U^\rho)$ and let
	$(\mathscr H^\rho, \nabla^\rho_{\on{GM}}) := (\mathbb W^\rho \otimes_{\mathbb C} \mathscr O_{\mathscr M},
	\on{id} \otimes d)$.
	
Because $\mathbb U^\rho$ is unitary, the Hodge-de Rham spectral sequence for
$R^1\pi^\circ_*(\mathbb{U}^\rho)$ 
degenerates \cite[Theorem
7.1(a)]{timmerscheidt:mixed-hodge-structure-for-unitary} (see also
\cite[Th\'eor\`eme 5.3.1]{saito1988modules} for a much more general result).
Hence there is a $2$-step Hodge filtration of 
$\mathscr H^\rho = \mathbb W^\rho \otimes \mathscr O_\mathscr{M}$
satisfying
\begin{equation}
	\label{equation:h-filtration}
\begin{aligned}
F^1 \mathscr H^\rho &\simeq \pi_* (\mathscr E^\rho \otimes
\Omega^1_{\mathscr{C}/\mathscr{M}}(\log \mathscr D)) \\
\mathscr H^\rho/ F^1 \mathscr H^\rho &\simeq R^1 \pi_* \mathscr E^\rho.
\end{aligned}
\end{equation}

In this paper, we will primarily be concerned with the weight $1$ part of
cohomology (essentially ignoring the part related to the punctures) and so we
now repeat the above construction for the weight $1$ part.
See, for example, 
\cite{deligne:hodge-ii} for background on the weight and hodge filtration,
\cite{timmerscheidt:mixed-hodge-structure-for-unitary}
for further background in the unitary case,
and \cite[\S4]{landesmanL:canonical-representations} for a concise description
of the weight and hodge filtration in the case of curves.

We now let $\mathbb V^\rho$ be the weight $1$ part of $\mathbb W^\rho$, which we can
explicitly identify as $R^1 \pi_*(j_* \mathbb U^\rho)$ by \cite[Theorem
4.1.1]{landesmanL:canonical-representations}.
Let 
\begin{align*}
	(\mathscr G^\rho, \nabla^\rho_{\on{GM}}) := (\mathbb V^\rho \otimes_{\mathbb C} \mathscr O_{\mathscr M}, \on{id} \otimes d).
\end{align*}
\end{notation}

To understand how the Hodge filtration of $\mathscr G^\rho$ interacts with the weight filtration, the following lemma
is key. The proof is a matter of unwinding definitions.
\begin{lemma}
	\label{lemma:weight-1-sections}
	Suppose $C$ is a smooth proper connected curve, $D \subset C$ is a
	divisor, and $\mathbb V$ is a
	local system on $C^\circ := C - D$ with unitary monodromy.
	Let $(E_\star, \nabla)$ be the Deligne canonical extension of $(\mathbb V
		\otimes_{\mathbb C} \mathscr O_{C^\circ}, \on{id}
	\otimes d)$ to $C$ as in, for example, \cite[Definition
	3.3.1]{LL:geometric-local-systems}.
	Under the natural isomorphism $F^1(H^1(C^\circ, \mathbb V)) \simeq H^0(C, E
	\otimes \omega_C(D))$, the weight $1$ part is given by
	\begin{align*}
	(W^1\cap F^1)(H^1(C^\circ, \mathbb V)) \simeq H^0(C, \widehat{E}_0
	\otimes \omega_C(D))\subset H^0(C, E
	\otimes \omega_C(D)).
	\end{align*}
\end{lemma}

\begin{proof}
	First, we claim $(W^1\cap F^1)(H^1(C^\circ, \mathbb V))$
	is identified with the subspace of
	$H^0(C, E \otimes \omega_C(D))$ vanishing under the residue map
	$H^0(C, E \otimes \omega_C(D)) \to i_* (\mathscr O_D \otimes
	E|_D)$, for $i: D \to C$ the inclusion. 
	Indeed, this follows from the
	degeneration of the Hodge-de Rham spectral sequence for unitary local
	systems \cite[Theorem
	7.1(a)]{timmerscheidt:mixed-hodge-structure-for-unitary}
	and the definition of the weight and Hodge filtrations as in
	\cite{timmerscheidt:mixed-hodge-structure-for-unitary}, culminating in
	\cite[Definition 6.1]{timmerscheidt:mixed-hodge-structure-for-unitary}.

	It only remains to identify
	\begin{align*}
		H^0(C, \widehat{E}_0 \otimes \omega_C(D)) = \ker \left( H^0(C, E \otimes_{\mathscr O_X} \omega_C(D)) \to i_* (\mathscr O_D \otimes
		E|_D) \right)
	\end{align*}
	We will explain why this equality
	is a matter of unwinding definitions.
	Specifically, we must unwind the definition of the Deligne canonical
	extension which yields the parabolic structure on $E$, as in
	\cite[Remarques 5.5(i)]{deligne:regular-singular}, (see also 
	\cite[Definition 3.3.1]{LL:geometric-local-systems} for a
summary in the relevant case) and the definition of $\widehat{E}_0$ 
	in \eqref{equation:coparabolic}.
	Indeed, using notation as in \autoref{subsection:parabolic-review},
	we recall $J$ denotes the set of integers $j \in \{ 1, \ldots, n\}$
	for which $\alpha^1_j = 0$. 
	For any $j \in J$,
	by definition of the Deligne canonical extension, 
	$E_j^2$ is the sum of the generalized eigenspaces of the residue map at
	$x_j$ with nonzero
	eigenvalue, as described in \cite[Definition
	3.3.1]{LL:geometric-local-systems}.
	Because the $0$-generalized eigenspace at $x_j$ is semisimple by
	assumption that the monodromy of $\mathbb V$ is unitary, it is identically $0$. 
	Hence, $\ker(E \to E_{x_j}/E_j^2)$ is the kernel of the residue
	map at $x_j$. It follows that $\widehat{E}_0 : = \ker( E \to \oplus_{j \in J}
	E_{x_j}/E_j^2)$
	is the kernel of the residue map along $D$, 
	and the analogous statement holds for $E$ replaced by $E \otimes
	\omega_C(D)$.
\end{proof}

Consequently we obtain a simple
description of the graded parts of the Hodge filtration of $\mathscr G^\rho$.
Since the weight $1$ part of the Hodge filtration surjects onto the second
graded piece of the weight filtration, \autoref{lemma:weight-1-sections}
yields the following description of the Hodge filtration of $\mathscr G^\rho$,
where we use notation as in 
\eqref{equation:relative-coparabolic}:
\begin{equation}
	\label{equation:g-filtration}
\begin{aligned}
F^1 \mathscr G^\rho &\simeq \pi_* ((\widehat{\mathscr E^\rho})_0 \otimes
\Omega^1_{\mathscr{C}/\mathscr{M}}(\log \mathscr D)) \\
\mathscr G^\rho/ F^1 \mathscr G^\rho &\simeq R^1 \pi_* \mathscr E^\rho.
\end{aligned}
\end{equation}

For $m \in \mathscr M$, let $\overline{\nabla}^\rho_m$ denote the fiber over $m$ of the map
\begin{align}
	\label{equation:derivative-period}
\overline{\nabla}^\rho: F^1 \mathscr G^\rho \hookrightarrow \mathscr G^\rho
\xrightarrow{\nabla^\rho_{\on{GM}}} \mathscr G^\rho \otimes
	\Omega^1_{\mathscr M} \twoheadrightarrow \mathscr G^\rho/F^1 \mathscr G^\rho \otimes
	\Omega^1_{\mathscr M}
\end{align}
where the first map is the natural inclusion and the last is the natural quotient map.

Using a version of the theorem of the fixed part, we now show how
counterexamples to the Putman-Wieland conjecture yield nonzero elements in the kernel of
$\overline{\nabla}^\rho_m$ (compare to \cite[Lemma 6.1.1]{landesmanL:canonical-representations}).
\begin{lemma}
	\label{lemma:period-map-kernel}
	Suppose $f: X \to Y$ furnishes a counterexample to Putman-Wieland.
	Let $\mathscr X^\circ \xrightarrow{\widetilde{f}^\circ} \mathscr C^\circ
	\xrightarrow{\pi^\circ} \mathscr M$ denote a
	versal family of $\phi$-covers, as in \autoref{definition:versal-phi},
	whose fiber over some $m\in \mathscr{M}$ is $f^\circ: X^\circ \to
	Y^\circ$ with regular compactification given by
	$f: X \to Y$.
Then there is some irreducible complex representation $\rho$ of $H$ such that $\overline{\nabla}^\rho_m$,
	as defined in \eqref{equation:derivative-period},
	has a nontrivial kernel.
	Moreover, if $f$ is $\rho$-isotypic in the sense of
	\autoref{definition:pw-counterexample}, then $\overline{\nabla}^\rho_m$
	vanishes.
\end{lemma}
\begin{proof}
	To begin, let us set up some notation.
	By assumption $f$ corresponds to a covering of topological surfaces
	$h: \Sigma_{g',n'} \to \Sigma_{g,n}$ 
	which furnishes a
	counterexample to Putman-Wieland.
	Then, there
	is some finite index $\Gamma' \subset \Gamma$ fixing a nonzero vector in
	$H^1(\Sigma_{g'}, \mathbb C)$.
	Because this action is tensored up from an action defined over the rational numbers, $\Gamma'$ also
	fixes a nonzero vector $v \in H^1(\Sigma_{g'}, \mathbb Q)$.
After replacing $\mathscr{M}$ by a finite \'etale cover, we may assume the
action of $\pi_1(\mathscr{M})$ on $H^1(\Sigma_{g'}, \mathbb Q)$ factors through
$\Gamma'$. Let $V=H^0(\mathscr{M}, R^1(\widetilde{f}^\circ \circ \pi^\circ)_*
\mathbb{Q})$; by the theorem of the fixed part \cite[14.52]{peters2008mixed}, 
$V$ has a natural mixed Hodge
structure compatible with its natural embedding $V\hookrightarrow
(R^1(\widetilde f^\circ \circ \pi^\circ)_* \mathbb{Q})_{x}$, for any $x\in
\mathscr{M}$. By assumption $V$ is nonzero, and moreover intersects the weight
$1$ part of $R^1(\widetilde f^\circ \circ \pi^\circ)_* \mathbb{Q})_{x}$, because our chosen vector fixed by $\Gamma'$ lies in
the weight $1$ part
$H^1(\Sigma_{g'}, \mathbb C) = W^1(H^1(\Sigma_{g',n'}, \mathbb C))$ of the
cohomology.

We next show that for some $\rho$, $\overline{\nabla}^\rho_m$ has a nontrivial kernel.
Let $j: Y^\circ \to Y$ denote the inclusion.
As $W^1V$ is a nonzero rational sub-mixed Hodge structure of the mixed Hodge structure on $H^1(Y^\circ, f^\circ_*
\mathbb{Q})= (R^1 (\widetilde{f}^\circ \circ \pi^\circ)_* \mathbb{Q})_m$, $W^1V$ has non-trivial
intersection with $(F^1\cap W^1)H^1(Y^\circ, f^\circ_*
\mathbb{Q})$, i.e. $(F^1\cap W^1)V$ is nonzero. Any element of $V= H^0(\mathscr{M}, R^1(\widetilde{f}^\circ \circ \pi^\circ)_*
\mathbb{Q}) $ lies in the
kernel of the Gauss-Manin connection by definition. Hence any element of $(F^1\cap W^1)V$ lies in the kernel of $$\overline{\nabla}_m: F^1H^1(Y, j_* f^\circ_* \mathbb C)\to H^1(Y, j_*f^\circ_* \mathbb C)/F^1H^1(Y, j_* f^\circ_* \mathbb C)\otimes T_{\mathscr{M}, m}^\vee.$$ Decomposing into $\rho$-isotypic pieces yields the
claim
that for some $\rho$, $\overline{\nabla}^\rho_m$ has a nontrivial kernel:
indeed, $R^1(\widetilde f^\circ \circ \pi^\circ)_*\mathbb{C}$ is a direct sum of copies of
$\mathbb{W}^\rho$, as defined in \autoref{notation:period-map},
by the Leray spectral sequence associated to the composition
$\mathscr X^\circ \xrightarrow{\widetilde{f}^\circ} \mathscr C^\circ
\xrightarrow{\pi^\circ} \mathscr M$.
Since this decomposition respects the weight filtration, we similarly find
$W^1(R^1(\widetilde f^\circ \circ \pi^\circ)_*\mathbb{C})$ is a direct sum of copies of
$\mathbb{V}^\rho$.

Finally, if $f$ is $\rho$-isotypic, this means that
$V$ contains all of $W^1(R^1 (\widetilde{f}^\circ \circ \pi^\circ)_* \mathbb{Q})^\rho$, and hence
$\overline{\nabla}_m^\rho$ vanishes.
\end{proof}

For the next statement, recall that a vector bundle $V$ on an integral variety $X$ is {\em generically
globally generated} if the evaluation map $H^0(X, V) \otimes \mathscr O_X \to V$
is a surjection over the generic point of $X$.
Recall also that for $\rho$ an $H$-representation and $X \to Y$ a Galois $H$-cover, the
associated map $\pi_1(Y) \to H \xrightarrow{\rho} \on{GL}_{\dim \rho}(\mathbb
C)$ yields a parabolic vector bundle $E^\rho_\star$ on $Y$ with underlying
vector bundle $E^\rho := E^\rho_0$ under the Mehta-Seshadri correspondence, as in
\autoref{notation:rep-to-vector-bundle}.

\begin{proposition}
	\label{proposition:non-ggg}
	Suppose $f: X \to Y$ 
	furnishes a counterexample to Putman-Wieland. 
	Then there exists an irreducible $H$-representation $\rho$ so that 
	the vector bundle
	$(E^\rho)^\vee \otimes \omega_Y$ on $Y$ is not generically globally generated.
\end{proposition}
\begin{proof}
By \autoref{lemma:period-map-kernel},
(with notation as in the statement of \autoref{lemma:period-map-kernel}), there exists $\rho$ such that
$\overline{\nabla}^\rho_m$ has a nontrivial kernel for any $m \in \mathscr M$.
Fixing $m \in \mathscr M$, let $Y := \mathscr C_m, E^\rho := (\mathscr
E_0^\rho)_m$ and $v \in \ker \overline{\nabla}^\rho_m$ nonzero.
Let $D \subset Y$ denote the branch divisor of $f: X \to Y$.
We can identify 
\begin{align*}
	v \in F^1 \mathscr G^\rho \simeq H^0(Y, (\widehat{E^\rho})_0 \otimes \omega_Y(D)) \subset H^0(Y, E^\rho \otimes \omega_Y(D))
\end{align*}
and think of $v$ as a nonzero map $\mu_v: (E^\rho)^\vee \otimes \omega_Y \to \omega_Y^{\otimes
2}(D)$.
By \cite[Theorem 5.1.6]{landesmanL:canonical-representations},
the vanishing of $\overline{\nabla}^\rho_m(v)$ is equivalent to the vanishing of the
induced map $H^0(Y, (E^\rho)^\vee \otimes \omega_Y) \to H^0(Y, \omega_Y^{\otimes
2}(D))$.
In other words, the map $\mu_v$ defined above induces the $0$ map on global sections. 
Therefore, all global sections of $(E^\rho)^\vee \otimes \omega_Y$
factor through the proper subbundle $\ker \mu_v$,
and so 
$(E^\rho)^\vee \otimes \omega_Y$ is not generically globally generated.
\end{proof}

\section{Isotypicity and non-unitarity}\label{section:bilinear-pairing}
The main result of this section is \autoref{theorem:isotypic}, which states that
counterexamples to Putman-Wieland in genus $\geq 2$ cannot be isotypic, i.e.,
there exists an element of 
$H^1(\Sigma_{g'}, \mathbb C)^\rho$ with infinite orbit under the action of a finite
index subgroup of the mapping class group. We show more, namely \autoref{theorem:non-unitary}: if $X\to Y$ is an $H$-cover, where $Y$ has genus at least $2$, the virtual action of the mapping class group of $Y$ on an $H$-isotypic component of the cohomology of $X$ is non-unitary.

In \autoref{corollary:boggi-looijenga} we use this to show how a
 result from the retracted paper of Boggi-Looijenga \cite{boggiL:curves-with-prescribed-symmetry} would imply the Putman-Wieland conjecture.

Our main tool for proving this is a natural bilinear pairing, which we
next introduce.
Let $C$ be a smooth proper connected curve of genus $g$, $D\subset C$ a reduced
divisor, $E_\star$
a parabolic vector bundle on $(C,D)$, and $E := E_0$. 
As described in \cite[(5.5)]{landesmanL:canonical-representations}, there is a
nondegenerate bilinear pairing
\begin{align}
	\label{equation:bilinear-pairing}
	B_E:
(E \otimes \omega_C(D))\times (E^\vee\otimes \omega_C) \to
\omega_C^{\otimes 2}(D)
\end{align}
given as the composition $$B_E: (E \otimes \omega_C(D))\times (E^\vee\otimes \omega_C)\overset{\otimes}{\longrightarrow} (E\otimes E^\vee)\otimes \omega_C^{\otimes 2}(D)\overset{\on{tr}\otimes \on{id}}{\longrightarrow} \omega_C^{\otimes 2}(D),$$
where $\on{tr}$ denotes the trace pairing ${E}\otimes {E}^\vee\to
\mathscr{O}_C.$ 
Note this pairing is nondegenerate, since on the fiber over $x$ it is obtained by the
pairing between the vector spaces $E_x$ and $E_x^\vee$.
By restriction to $H^0(C, \widehat{E}_0 \otimes \omega_C(D)) \subset H^0(C, E \otimes
\omega_C(D))$, we also obtain an induced pairing
\begin{align*}
	H^0(C, \widehat{E}_0 \otimes \omega_C(D)) \times H^0(C, E^\vee \otimes
	\omega_C) \to H^0(C, \omega^{\otimes 2}_C(D)).
\end{align*}

\begin{theorem}
	\label{theorem:isotypic-derivative}
	Let $E_\star$ be a semistable parabolic bundle on $(C,D)$ of parabolic
	degree zero, with underlying vector bundle $E := E_0$. 
	Suppose $g \geq 2$. Then the pairing $H^0(\widehat{E}_0 \otimes \omega_C(D)) \otimes H^0(E^\vee
	\otimes \omega_C) \to H^0(C, \omega_C^{\otimes 2}(D))$ cannot vanish.
\end{theorem}
\begin{proof}
	First, we may assume $\rk E > 1$ by \cite[Proposition
	5.2.3]{landesmanL:canonical-representations}, as $g \geq 2$.
	If $F_\star$ is a parabolic sheaf, we call the image of $H^0(C, F_0)
	\otimes \mathscr O_C \to F_0 \subset F_\star$ the globally generated
	subsheaf of
	$F_\star$, which is by definition also a subsheaf of $F_0$.
	Let $U \subset \widehat{E}_\star \otimes \omega_C(D)$ denote the globally generated
	subsheaf of $\widehat{E}_\star \otimes \omega_C(D)$ and
	$V$ denote the globally generated subsheaf  
	of $\reallywidehat{((E_\star)^\vee \otimes \omega_C(D))}_0$.
	Under the identifications 
\begin{align*}
	\left(\widehat{E}_\star \otimes \omega_C(D)\right)_0 &= \widehat{E}_0 \otimes \omega_C(D) \\
\reallywidehat{((E_\star)^\vee \otimes \omega_C(D))}_0 &= E^\vee \otimes
\omega_C,
\end{align*}
(see \cite[Definition 2.6.1]{LL:geometric-local-systems} for the notion of a dual of a parabolic bundle)
$U$ and $V$ are also the globally generated subsheaves of
$\widehat{E}_0 \otimes \omega_C(D)$ and $E^\vee \otimes \omega_C$ respectively.

Let $c_V := \rk E - \rk V$ and $c_U := \rk E - \rk U.$
Using \autoref{proposition:generic-parabolic-global-generation}
applied to the bundles
$E_\star \otimes \omega_C(D)$ and
$E_\star^\vee \otimes \omega_C(D)$,
we know $\rk E \geq g c_V$ and $\rk E \geq g c_U$.

On the other hand, we claim $\rk U + \rk V \leq \rk E$.
Granting this claim, we find $c_V + c_U \geq \rk E$.
Since $\rk E \geq g c_V$ and $\rk E \geq g c_{U}$, adding these gives
\begin{align}
	\label{equation:equality-case}
	2 \rk E \geq g(c_V + c_U) \geq g \rk E,
\end{align}
implying $2 \geq g$.

In the case $g > 2$, it remains to show $\rk U + \rk V \leq \rk E$.
We will argue this using the fact that an isotropic subsheaf for a non-degenerate quadratic
form
on a vector bundle of rank $2\rk E$ has rank at most $\rk E$.
Indeed, consider the quadratic form $q_E$ on $E \otimes \omega_C(D) \oplus
E^\vee \otimes \omega_C$ associated to the nondegenerate bilinear form $B_E$ of
\eqref{equation:bilinear-pairing}:
\begin{align*}
	q_E : E \otimes \omega_C(D) \oplus
	E^\vee \otimes \omega_C &\to \omega_C^{\otimes 2}(D) \\
	(v,w) &\mapsto B_E(v,w).
\end{align*}
Any vector bundle subsheaf of $E \otimes \omega_C(D) \oplus E^\vee \otimes \omega_C$ 
isotropic for this quadratic form has rank at most $\rk E$, as may be verified
on the generic fiber using that isotropic subspaces of a rank $2 \rk E$
non-degenerate quadratic space have dimension at most $\rk E$.
Therefore, it is enough to
show $U \oplus V$ is killed under $q_E$.
Using $q_E(U \oplus V) = B_E(U \times V)$,
it is enough to show $B_E(U \times V) = 0$.
We have a commutative diagram
\begin{equation}
	\label{equation:}
	\begin{tikzcd} 
		(H^0(C, U) \otimes \mathscr O_C)\times  (H^0(C, V) \otimes \mathscr O_C) \ar {r} \ar {d}
		& H^0(C, \omega_C^{\otimes 2}(D)) \otimes \mathscr O_C \ar {d} \\
		U \times V \ar {r}{B_E} & \omega_C^{\otimes 2}(D)
\end{tikzcd}\end{equation}
where the top horizontal map vanishes by assumption.
Since the vertical maps are surjective, the bottom horizontal map satisfies $B_E(U \times
V) = 0$, as desired.

To conclude, we also rule out the case $g = 2$.
If $g = 2$, we must have equality in \eqref{equation:equality-case}, which
forces $\rk E = 2 c_U$.
This means we have equality in
\autoref{proposition:generic-parabolic-global-generation}(II),
which is proved in \cite[Proposition
6.3.6]{LL:geometric-local-systems}.
If equality holds, we also have equality in
\cite[Lemma 6.2.3]{LL:geometric-local-systems},
which means 
$\mu(\widehat{E}_0 \otimes \omega_C(D)) = 2g - 2$
(where we are taking the bundle named $V$ in 
\cite[Lemma 6.2.3]{LL:geometric-local-systems}
to be $\widehat{E}_0 \otimes \omega_C(D)$).
This means
$E_\star$ has trivial parabolic structure at each parabolic point,
so $\widehat{E}_0 \otimes
\omega_C(D) = E \otimes \omega_C$.
If we had $\rk E = 2 c_U$, we would have $H^0(C, E \otimes \omega_C) = H^0(C,
U)$, which means $h^0(C, U) = 2 \rk U$ and so $h^1(C, U) = \rk U$.
By Clifford's theorem for vector bundles, as in \cite[Lemma
6.2.3]{LL:geometric-local-systems},
this can only happen when $\mu(U) = 2g - 2$, in which case we would have $U = E
\otimes \omega_C$. This equality contradicts the assumption that $\rk E = 2 c_U$, i.e.
that $\rk E = 2 \rk U$.
\end{proof}

\subsection{}
\label{subsection:isotypic-proof}

We now deduce our desired isotypicity consequence, \autoref{theorem:isotypic} and \autoref{theorem:non-unitary}, 
for the Putman-Wieland conjecture.
With setup as in \autoref{conjecture:putman-wieland-intro},
we have a cover $\Sigma_{g',n'} \to \Sigma_{g,n}$ and 
an action of a finite index subgroup $\Gamma \subset \on{Mod}_{g,n+1}$
on $H^1(\Sigma_{g',n'}, \mathbb C)$. We aim to show that if $\rho$ is any
irreducible $H$ representation
so that every element of the
characteristic subspace
$H^1(\Sigma_{g'}, \mathbb C)^\rho \subset H^1(\Sigma_{g',n'}, \mathbb C)$ has
finite orbit under $\Gamma$,
then $g \leq 2$.

\begin{proof}[Proof of \autoref{theorem:isotypic} and \autoref{theorem:non-unitary}]
We first prove \autoref{theorem:non-unitary}.

	Assume to the contrary that the virtual action of  
the mapping class group on $H^1(\Sigma_{g'}, \mathbb
	C)^\rho$ is unitary. Let 
	$\mathscr{X}^\circ\xrightarrow{\widetilde{f}^\circ}
	\mathscr{C}^\circ\xrightarrow{\pi^\circ} \mathscr{M}$ be a versal
	family of $H$-covers, and $\mathscr{X}\to \mathscr{C}\to \mathscr{M}$ the associated families of proper curves. Let $m\in \mathscr{M}$ be a point, and $X\to Y$ the fiber over $\mathscr{M}$ 
	Let $E^\rho_\star$ be the parabolic bundle on $Y$ corresponding to the
	representation $\rho$ as in \autoref{notation:rep-to-vector-bundle}.
	
	We claim the map
	$\overline\nabla_m^\rho$ vanishes identically, as $W^1R^1\pi^\circ_*\rho$, the variation of Hodge structure on $\mathscr{M}$ associated to $H^1(X, \mathbb
	C)^\rho$, is unitary by assumption. Indeed, by
	\cite[1.13]{deligne1987theoreme}, we may write
	$$W^1R^1\pi^\circ_*\rho=\bigoplus_i \mathbb{V}_i\otimes W_i,$$ where the
	$\mathbb{V}_i$ are complex variations of Hodge structure with
	irreducible monodromy and the $W_i$ are constant variations. The
	$\mathbb{V}_i$ carry a unique structure of a $\mathbb{C}$-VHS, up to
	renumbering. As $W^1R^1\pi^\circ_*\rho$ is unitary by assumption, the
	same is true of each $\mathbb{V}_i$, and so the Hodge filtration of each
	$\mathbb V_i$ has length at most one. This proves the claim that
	$\overline\nabla_m^\rho = 0$.
	
Note that $\overline{\nabla}^\rho_m$ is the weight $1$ part of
the map adjoint to the multiplication map
\begin{align*}
	H^0(Y, E^\rho_0 \otimes \omega_Y(D)) \otimes H^0(Y, (E^\rho_0)^\vee
	\otimes \omega_Y) \to H^0(Y, \omega_Y^{\otimes 2}(D))
\end{align*}
	by \cite[Theorem 5.1.6]{landesmanL:canonical-representations}.
	Since the 
	subspace of $H^0(Y, E^\rho_0 \otimes \omega_Y(D))$ corresponding to $W^1\cap F^1$
	is $H^0(Y, \widehat{E}_0^\rho \otimes \omega_Y(D))$ by
	\autoref{lemma:weight-1-sections},
	we obtain that the multiplication map
	\begin{align*}
		H^0(Y, \widehat{E}_0^\rho \otimes \omega_Y(D)) \otimes H^0(Y, (E^\rho_0)^\vee
	\otimes \omega_Y) \to H^0(Y, \omega_Y^{\otimes 2}(D))
	\end{align*}
	also vanishes.
	Using \autoref{theorem:isotypic-derivative}, this implies $g \leq 2$.
	
	\autoref{theorem:isotypic} is immediate, as representations with finite image are unitary.
\end{proof}

As a consequence, we show how a claimed result from a paper of Boggi-Looijenga (which has since been retracted by the authors) implies the
Putman-Wieland conjecture.

\begin{corollary}
	\label{corollary:boggi-looijenga}
	Suppose \cite[Theorem
	B(i)]{boggiL:curves-with-prescribed-symmetry} were
	true.
	Then the Putman-Wieland conjecture,
	\autoref{conjecture:putman-wieland-intro}, would hold for all
	 $g \geq 3$.
\end{corollary}
\begin{remark}
	\label{remark:}
	We note that, unfortunately, there is a fatal error in the
	proof of \cite[Theorem
	B(i)]{boggiL:curves-with-prescribed-symmetry},
	appearing in \cite[Lemma 1.6]{boggiL:curves-with-prescribed-symmetry}
	(which is used in the proof of 
\cite[Theorem 1.1]{boggiL:curves-with-prescribed-symmetry},
and hence of
\cite[Theorem B(i)]{boggiL:curves-with-prescribed-symmetry}).
The error comes in the penultimate sentence of the proof of \cite[Lemma
1.6]{boggiL:curves-with-prescribed-symmetry},
where is it claimed
that ``It follows\dots,'' but in fact no argument is given for this claim. It is for this reason that the authors of \cite{boggiL:curves-with-prescribed-symmetry} retracted that paper.
\end{remark}

\begin{proof}
	Suppose $g \geq 3$, and we are
	given a finite \'etale $H$-cover of Riemann surfaces
	$f: \Sigma_{g',n'} \to \Sigma_{g,n}$ furnishing a counterexample to
	Putman-Wieland.
	This means some 
	subrepresentation $\chi \subset H^1(X, \mathbb
	C)$ has finite orbit under the action of the mapping class group
	$\on{Mod}_{g,n+1}$.
	The isotypicity statement of 	
	\cite[Theorem B(i)]{boggiL:curves-with-prescribed-symmetry}
	implies that if $\chi \subset H^1(X, \mathbb
	C)$ has finite orbit under the action of the mapping class group
	$\on{Mod}_{g,n+1}$, 
	every element of the $\chi$-isotypic component
	$H^1(X, \mathbb C)^\chi$ has finite orbit under the action of 
	$\on{Mod}_{g,n+1}$.	
	By definition, this means
	$f: \Sigma_{g',n'} \to \Sigma_{g,n}$ is $\chi$-isotypic
	in the sense of \autoref{definition:pw-counterexample},
	which contradicts \autoref{theorem:isotypic}.
\end{proof}

\section{Deducing a result of Markovi\'c-To\v{s}i\'c}
\label{section:markovic-tosic}

Our goal in this section is to show how our results toward the Putman-Wieland conjecture 
yield an alternate proof of \cite[Theorem 1.5]{markovic2}, which gives
constraints on any counterexample to Putman-Wieland.
Specifically, we will prove \autoref{theorem:markovicT}, which is slightly
stronger than \cite[Theorem 1.5]{markovic2}, as explained in
\autoref{remark:slight-generalization}.

We now briefly outline the structure of the proof.
We will deduce \autoref{theorem:markovicT} from \autoref{proposition:gonality}
using the Li-Yau inequality, following the idea in
\cite{markovic2}.
Hence, we focus on explaining
\autoref{proposition:gonality}, which shows that any counterexample to
Putman-Wieland has low gonality.
To verify this claim about gonality, we 
recall from \autoref{proposition:non-ggg}
that any counterexample to Putman-Wieland $f: X \to
Y$ yields a representation $\rho$ so that $(E^\rho)^\vee \otimes \omega_Y$ is
not generically globally generated.
With this in mind, we give some equivalent criteria for this non-generic global
generation in \autoref{lemma:prill-equivalence}.
Using \autoref{lemma:prill-equivalence}, we obtain cohomological bounds in
\autoref{lemma:prill-large-sections}, which then easily implies the desired
bound for gonality of \autoref{proposition:gonality}.

Hence, our first goal is to describe equivalent conditions for generic global
generation of $(E^\rho)^\vee \otimes \omega_Y$
in \autoref{lemma:prill-equivalence}. 
For this, it will be useful to reduce to the case $f: X
\to Y$ is Galois, in which the following decomposition of $f_* \mathscr O_X$ 
will be useful.
This result is well known to experts, and we include it for completeness.

\begin{lemma}
	\label{lemma:structure-decomposition}
	Let $f: X \to Y$ be a Galois $H$-cover, let $\rho$ be an irreducible
	$H$-representation, and let 
	$E^\rho_\star$ denote the parabolic bundle associated to $\rho$ as in
	\autoref{notation:rep-to-vector-bundle}, and let $E^\rho := E^\rho_0$.
	Then 
	$f_* \mathscr O_X \simeq \oplus_{\text{$H$-irreps $\rho$}} (E^\rho)^{\oplus \dim \rho}$.
\end{lemma}
\begin{proof}
In \autoref{lemma:regular-rep} below, we show 
$f_* \mathscr O_X \simeq E^{\on{reg}}_0$.
Therefore, it suffices to show $E^{\on{reg}}_0 \simeq \oplus_{\text{$H$-irreps $\rho$}} (E^\rho)^{\oplus \dim \rho}$.
Since the regular representation of $H$ decomposes as 
$\on{reg} \simeq \oplus_{\text{$H$-irreps $\rho$}} \rho^{\oplus \dim \rho}$,
the Mehta-Seshadri correspondence \cite{mehta1980moduli} gives the desired isomorphism
$E^{\on{reg}}_0 \simeq \oplus_{\text{$H$-irreps $\rho$}} (E^\rho)^{\oplus \dim \rho}$.
\end{proof}

\begin{lemma}
	\label{lemma:regular-rep}
	Let $f: X \to Y$ be a Galois $H$-cover.
Let $E^{\on{reg}}_\star$ denote the parabolic bundle on $Y$ associated to the
	regular $H$ representation under the Mehta-Seshadri correspondence.
	Then, $E^{\on{reg}}_0 \simeq f_* \mathscr O_X$.
\end{lemma}
\begin{proof}
	Let $Y^\circ \subset Y$ be the locus on which $f$ is
\'etale, let $f^\circ : X^\circ := f^{-1}(Y^\circ) \to Y^\circ$ be the restriction of $f$ to
$Y^\circ$,
and let $j: Y^\circ \to Y$ be the inclusion. 
We find
\begin{align*}
	E^{\on{reg}_{Y^\circ}} \simeq E^{f^\circ_* \on{triv}_{X^{\circ}}} \simeq
	f^\circ_* E^{\on{triv}_{X^\circ}} \simeq f^\circ_* \mathscr O_{X^\circ}.
\end{align*}
By adjunction, this gives us a map
\begin{align*}
	E^{\on{reg}_{Y}} \to j_* j^* E^{\on{reg}_{Y}} \simeq j_* E^{\on{reg}_{Y^\circ}}
	\simeq j_* f^\circ_* \mathscr O_{X^{\circ}} \simeq f_* \mathscr O_X.
\end{align*}
Since this is an isomorphism over $Y^\circ$, it is enough to verify it is an
isomorphism in a neighborhood of each point of $Y - Y^\circ$, which follows from a local
computation.
\end{proof}

Using the above decomposition in the Galois case, we now give several
reformulations of non-generic global generation of $E^\rho \otimes \omega_Y$.
For the first, we recall a definition of Prill exceptional covers, named after
their relation to Prill's problem, as discussed in \cite{landesmanL:prill}, see
also \autoref{remark:prill-solution} and \autoref{remark:companion-paper}.

\begin{definition}
	\label{definition:}
	A finite cover $f: X \to Y$ of smooth proper geometrically connected
curves is {\em Prill exceptional} if 
$h^0(X, \mathscr O_X(f^{-1}(y))) \geq 2$ for every point $y \in Y$.
\end{definition}

\begin{lemma}
	\label{lemma:prill-equivalence}
	Let $f: X \to Y$ be a finite cover of smooth proper connected curves
	whose Galois closure has Galois group $H$.
	The following are equivalent:
	\begin{enumerate}
		\item The map $f$ is Prill exceptional.
		\item There is some irreducible nontrivial $H$-representation $\rho$ for which the
	associated vector bundle $E^\rho := E^\rho_0$ 
	as in \autoref{notation:rep-to-vector-bundle} is a summand of $f_*
	\mathscr O_X$ and 
	$h^0(Y, E^\rho(p)) > 0$	for a general $p \in Y$.
\item For the same $\rho$ as in the previous part, $(E^\rho)^\vee \otimes \omega_Y$ is not generically globally generated.
	\end{enumerate}
	\end{lemma}
\begin{proof}
	We first show the equivalence of $(1)$ and $(2)$.
	In the case $f$ is a Galois cover, we can apply
the decomposition 
$f_* \mathscr O_X \simeq \oplus_{\text{$H$-irreps $\rho$}} (E^\rho)^{\oplus \dim \rho}$
from \autoref{lemma:structure-decomposition}.
We find 
\begin{align*}
	h^0(X, \mathscr O_X(f^{-1}(p)) = h^0(Y, (f_* \mathscr O_X)(p))
		=\oplus_{\text{$H$-irreps $\rho$}} h^0(Y, (E^\rho)^{\oplus \dim
		\rho}(p)).
\end{align*}
Therefore, in the Galois case, $f$ is Prill exceptional if and only if there is some
$E^\rho$, a summand of $f_* \mathscr O_X$, with 
$h^0(Y, E^\rho(p)) > 0$ for a general $p \in Y$.

We next verify the equivalence of $(1)$ and $(2)$ in the case that $f$ is not
Galois.
Let $X'\xrightarrow{h} X \xrightarrow{f} Y$ denote its Galois closure, so that $f \circ h$ has
	Galois group $H$.
Since we are working in characteristic $0$, 
$\mathscr O_X$ is a
summand of $h_* \mathscr O_{X'}$ and therefore $f_* \mathscr O_X$ shows
up as a summand in $(f \circ h)_* \mathscr O_{X'}$.
Therefore, we can decompose $f_* \mathscr O_X$ as a sum 
$f_* \mathscr O_X \simeq \oplus_{\text{$H$-irreps $\rho$}} 
(E^\rho)^{\oplus a_\rho}$ where $0 \leq a_\rho \leq \dim \rho$.
As in the Galois case, we again find 
$f$ is Prill exceptional if and only if there is some
$E^\rho$ with 
$h^0(Y, E^\rho(p)) > 0$ for a general $p \in Y$.

We conclude by demonstrating the equivalence of $(2)$ and $(3)$.
Let $F^\rho := (E^\rho)^\vee \otimes\omega_Y$.
We wish to show $F^\rho$ is not generically globally generated if and only
if $h^0(Y, E^\rho(p)) > 0$.
Note that $F^\rho$
is not generically globally generated if and only if, for a general $p \in Y$, we have an exact sequence
\begin{equation}
	\label{equation:}
	\begin{tikzcd}
		0 \ar {r} & H^0(Y,F^\rho(-p)) \ar {r} & H^0(Y, F^\rho) \ar {r} &
		H^0(Y, F^\rho|_p)
\end{tikzcd}\end{equation}
which is not right exact.
Since $F^\rho|_p$ is supported at $p$, $h^0(Y, F^\rho|_p) = \rk F^\rho$,
and so failure of right exactness is equivalent to
$h^0(Y, F^\rho(-p)) >h^0(Y, F^\rho) - \rk F^\rho$.
Note that because $\rho$ is nontrivial, $0 = h^0(E^\rho) = h^1(Y, F^\rho).$
Therefore, using Riemann-Roch, 
$h^0(Y, F^\rho(-p)) >h^0(Y, F^\rho) - \rk F^\rho$
is equivalent to 
$h^1(Y, F^\rho(-p))>  h^1(Y, F^\rho) = 0$.
By Serre duality, this is equivalent to
$h^0(Y, E^\rho(p)) > 0$, as we wanted to show.
\end{proof}

Combining 
\autoref{proposition:non-ggg} and 
\autoref{lemma:prill-equivalence}, we easily deduce that counterexamples to
Putman-Wieland yield Prill exceptional covers.
\begin{lemma}
	\label{lemma:pw-gives-prill}
	If $X$ has genus $g'$, $Y$ has genus $g$, and $f: X \to Y$ is a cover
	furnishing a counterexample to Putman-Wieland, then
	$f: X \to Y$ is Prill exceptional.
\end{lemma}
\begin{proof}
	Suppose $f: X \to Y$ is a Galois $H$-cover which
	furnishes a counterexample to Putman-Wieland.
	We saw in \autoref{proposition:non-ggg} that
	there is some irreducible $H$-representation $\rho$ so that
	$(E^\rho)^\vee \otimes \omega$ is not generically globally generated.
	By \autoref{lemma:prill-equivalence}, $f$ is Prill exceptional.
\end{proof}

The following remarks will not be needed in what follows.
We include them as a pleasant, immediate application of the work we
have done so far.

\begin{remark}
	\label{remark:prill-solution}
	Prill's problem, \cite[p. 268, Chapter VI, Exercise D]{ACGH:I},
	asks whether any curve of genus $g \geq 2$ has a Prill exceptional
	cover.
	By \autoref{lemma:pw-gives-prill}, any counterexample to Putman-Wieland
	yields a Prill exceptional cover. 
	Since 
\cite[Theorem 1.3]{markovic} gives $f: X \to Y$ which is a 
counterexample to Putman-Wieland in genus $2$, we find that a general
	curve of genus $2$ has a Prill exceptional cover.

\end{remark}
\begin{remark}
	\label{remark:companion-paper}
	In a companion paper, \cite{landesmanL:prill},
	we prove a stronger result than that mentioned in 
	\autoref{remark:prill-solution}.
	Namely, we show that {\em any} smooth proper connected
	curve of genus $2$ over the complex numbers has a Prill
	exceptional cover, as opposed to just a {\em general} curve of genus $2$.
	We include \autoref{remark:prill-solution} in the present paper as it is
	an immediate consequence of the work we do here to obtain an alternate
	proof of the
	result of
	Markovi\'c-To\v{s}i\'c.
	However, we chose to write \cite{landesmanL:prill} as a separate paper
	in order to give a mostly self-contained exposition of Prill's problem, 
	which is easier to follow than the above
	construction of Prill exceptional covers. 
	In particular, it is not mired in the
	notation of parabolic bundles.
\end{remark}

Another consequence of the above results is the following bound on sections 
associated to the divisor given by a fiber of $f$, which will be the key to
verifying a gonality estimate to deduce the result of Markovi\'c-To\v{s}i\'c,
\autoref{theorem:markovicT}.

\begin{lemma}
	\label{lemma:prill-large-sections}
	Any 
	Prill exceptional cover $f: X \to Y$ 
satisfies 
$h^0(X, \mathscr O_X(f^{-1}(p)))\geq 2$.
If $f$ is Galois, we moreover have
$h^0(X, \mathscr O_X(f^{-1}(p)))\geq g + 1$.
\end{lemma}
\begin{proof}
	The non-Galois case follows from the definition of a Prill exceptional
	cover.

	We next show, in the case that $f$ is Galois, that
$\dim \rho \geq g$.
Let $D$ denote the branch divisor of $f$, and let $n := \deg D$.
We wish to apply
\autoref{proposition:generic-parabolic-global-generation}
to the bundle
$E_\star = E^{\rho^\vee}_\star \otimes \omega_Y(D)$
and so we next verify its hypotheses.
If $f$ is Galois,
$(E^\rho)^\vee \otimes\omega_Y$ is not generically
globally generated by
\autoref{lemma:prill-equivalence}. 
By 
\cite[(3.1)]{yokogawa:infinitesimal-deformation} we have $\widehat{E}^{\rho^\vee}_0(D)
\simeq (E^{\rho}_0)^\vee$.
This implies $\widehat{E}^{\rho^\vee}_0 \otimes \omega_Y(D)
\simeq (E^{\rho}_0)^\vee \otimes \omega_Y$.
Hence, the bundle $\widehat{E}^{\rho^\vee}_\star \otimes \omega_Y(D)$
is not generically globally generated.
We also have 
\begin{align*}
	\mu_\star(E^{\rho^\vee}_\star \otimes \omega_Y(D)) =
	\mu_\star(E^{\rho^\vee}_\star) + \deg \omega_Y(D) = 0 + 2g-2 + n = 2g -
	2+ n.
\end{align*}
Therefore, we may apply the final statement of
\autoref{proposition:generic-parabolic-global-generation}
to the parabolic bundle $E_\star = E^{\rho^\vee}_\star \otimes \omega_Y(D)$
to conclude $\dim \rho \geq g$.

Combining the above observations, we have
\begin{align*}
h^0(X, \mathscr O_X(f^{-1}(p))) &= h^0(Y, (f_* \mathscr O_X)(p)) \\
	&\geq h^0(Y, E^{\on{triv}}_0(p) \oplus (E^\rho_0)^{\oplus \dim \rho}(p)) \\
	&\geq h^0(Y, E^{\on{triv}}_0(p) \oplus (E^\rho_0)^{\oplus g}(p)) \\
	&\geq 1 + g.\qedhere
\end{align*}
\end{proof}

\subsection{}
\label{proof:gonality}
We next prove \autoref{proposition:gonality}, regarding the gonality of
counterexamples to Putman-Wieland.
\begin{proof}[Proof of \autoref{proposition:gonality}]
	We first want to show that if $f: X \to Y$ furnishes a counterexample to Putman-Wieland, then
$\on{gon}(X) \leq \deg f$.
In the case $f: X \to Y$ is Galois,
we wish to show $X$ has gonality at most $\deg f - (g-1)$.

Note first that $f: X \to Y$ is Prill exceptional by
\autoref{lemma:pw-gives-prill}.
Hence, we may apply \autoref{lemma:prill-large-sections}, which implies that for a general point $p \in Y$, we have $h^0(X,
	\mathscr O(f^{-1}(p))) \geq 2$, and moreover 
$h^0(X,	\mathscr O(f^{-1}(p))) \geq g+1$
	when $f$ is Galois.
In the general case, $\mathscr O(f^{-1}(p))$ is a line bundle of degree $\deg f$
with $2$ global sections, meaning $X$ has gonality at most $\deg f$.

We now assume $f$ is Galois and aim to show
$\on{gon}(X) \leq \deg f - (g-1)$.
Choose a generic degree $g-1$ divisor $S$ on $X$. Then 
\begin{align*}
h^0(X, \mathscr O_X(f^{-1}(p)-S)) \geq 
h^0(X, \mathscr O_X(f^{-1}(p))) - \deg S
=
1 + g -(g-1) = 2,
\end{align*}
and so $f^{-1}(p)-S$ is a divisor of degree $\deg f - (g-1)$ which still has a
$2$-dimensional space of global sections. Therefore, $X$ has gonality at most $\deg
f - (g-1)$, as claimed.
\end{proof}

To conclude the paper, we deduce \autoref{theorem:markovicT} from
\autoref{proposition:gonality}.
For this, we will connect nonzero eigenvalues of the Laplacian of a curve
$X$ to the gonality of $X$, 
following the idea of \cite{markovic2}.
The key input in \cite[Theorem 1.5]{markovic2} is their verification in the end
of \cite[\S8]{markovic2} that if $f: X \to Y$ furnishes a counterexample to
Putman-Wieland, then the gonality of $X$ is at most $\deg f$.
We prove prove a slightly stronger statement, using 
different methods.
Following the idea of \cite{markovic2}, we now explain why the gonality estimate
of 
\autoref{proposition:gonality} implies \autoref{theorem:markovicT}, using the
Li-Yau inequality. 

\subsection{}
\label{proof-mt}
\begin{proof}[Proof of \autoref{theorem:markovicT}]
Let $X$ be a curve of genus $g'$, $Y$  a curve of genus $g$ and $f:
X \to Y$ a cover furnishing a counterexample to Putman-Wieland.
Recall we are trying to show that if 
$\lambda_1(X)$ denotes the smallest nonzero eigenvalue of the Laplacian acting on
$X$, then $\frac{1}{g-1} \geq 2 \lambda_1(X)$, and moreover there is a strict
inequality when $f$ is Galois.

It follows from the Li-Yau inequality \cite[Theorem
	1]{liY:a-new-conformal-invariant}, 
	as explained in \cite[(11)]{ellenbergHK:expander-graphs},
	that $\on{gon}(X) \geq 2\lambda_1(X)(g' - 1)$.
The statement of
\autoref{proposition:gonality}
tells us that any counterexample to Putman-Wieland, $f: X \to Y$,
satisfies
$\deg f \geq \on{gon}(X)$.
Combining this with the above consequence of the Li-Yau inequality and
Riemann-Hurwitz yields
\begin{align*}
	\deg f \geq \on{gon}(X)
	\geq 2 \lambda_1(g'-1)
	\geq 2 \lambda_1(X) \deg f (g-1).
\end{align*}
Dividing both sides by $\deg f(g-1)$ gives
$\frac{1}{g-1} \geq 2 \lambda_1(X)$.

In the case $f$ is Galois, since $g \geq 2$, we get a strict inequality 
$\deg f > \on{gon}(X)$, 
from \autoref{proposition:gonality},
which similarly implies 
$\frac{1}{g-1} > 2 \lambda_1(X)$.
\end{proof}

\bibliographystyle{alpha}
\bibliography{bibliography-mcg-hodge-theory}

\end{document}